\documentclass[12pt]{article}
\usepackage{amsmath, amssymb, amsfonts, amsthm}
\usepackage{enumitem}
\usepackage{multicol}
\usepackage{hyperref}
\usepackage{tikz}
\usepackage{array}
\usepackage{tabularx}
\tikzstyle{vertex}=[circle, draw, inner sep=0pt, minimum size=6pt]

\allowdisplaybreaks
\def \N {{\mathbb{N}}}

\def \ga {{\gamma}}

\newtheorem*{theorem*}{Theorem}
\newtheorem{theorem}{Theorem}

\newtheorem{lemma}[theorem]{Lemma}

\newtheorem{definition}[theorem]{Definition}

\newtheorem{rem}[theorem]{Remark}
\newtheorem*{ex*}{Example}

\title{$\gamma$-Chiral is same as Chiral}
\author {
Shrinit Singh\\
 Department of 
 Mathematics, Shiv Nadar Institution of Eminence, India-201314\\ (e-mail: 
 ss101@snu.edu.in).}
\date{}
\begin{document}
\maketitle
\begin{abstract}
A word $w$ in a free group is called {\em chiral} if there exists a group $G$ such that image of word map corresponding to word $w$ is not closed with respect to inverse. Similarly a group $G$ is said to be {\em chiral}  if there exists a word $w$ in free group such that $w$ exhibits chirality on the group $G$. Gordeev et al.  \cite{gordeev2018geometry} extended the concept of chirality to introduce $\gamma$-chirality in both cases. We show that the notion of $\gamma$-chirality is equivalent to chirality.
\end{abstract}
{\bf{Key Words}}: Free group, Word map, Chiral word \\
{\bf{AMS(2020)}}:20F10 \\

\section*{Introduction}
Let $F_d$ be a free group with a basis $x_1, \ldots, x_d$. By a word $w \in F_d$, we mean a reduced word in $F_d.$
Each word $w\in F_d$ defines a map from $G^d$ to $G$ which is an evaluation map. For example, if $w=x_1x_3^2\in F_3,$ then $w:G^3\to G$ defined as $w((g_1,g_2,g_3))=g_1g_3^2.$ We denote the image of the word map  $w(G^d)$ as $G_w.$ 

Let $G$ be a group. An anti-automorphism $\delta$ of the group $G$ is a bijective map from $G$ to $G$ such that $\delta(ab) = \delta(b)\delta(a)$ for all $a,b \in G$.
Let $AA(G)$ be the set all anti-automorphisms of $G$ and $A(G)$ be the set of all automorphisms of $G.$ Then it is easy to see that there is a one to one correspondence between $A(G)$ and $AA(G).$ In particular, if $f\in A(G),$ then $g:G\to G$ defined as $g(x)=f(x^{-1})$ belongs to $AA(G)$ and vice versa. 

 In \cite{gordeev2018geometry}
Gordeev et al. generalized the concept of chirality  to  $\gamma$-chirality in following way.    
\begin{definition}[\cite{gordeev2018geometry}]
 \begin{enumerate}
  \item A word $w\in F_d$ is a {\em $\gamma$-chiral}, where $\gamma \in AA(F_d)$, if there exists a group $G$ such that $G_w\ne G_{\gamma(w)}.$ In this case, chirality of word $w$ is a special case of $\gamma$-chirality for $\gamma(u) = u^{-1}$ for all $u \in F_d.$  
  \item Let $G$ be a group, and let $\gamma \in AA(G)$. For every word $w \in F_n$, define $w_{\gamma} : G^{(n)} \rightarrow G$ by $w_{\gamma}(g_1, \cdots , g_n) = \gamma(w(g_1, \cdots , g_n))$. Then $G$ is said to be $\gamma$-chiral if there exists a word $w \in F_d$ for some $d$ such that the images of $w$ and $w_{\gamma}$ are different. In this case also, chirality of group $G$ is a special case of $\gamma$-chirality for $\gamma(g)=g^{-1}$ for all $g \in G$.
 \end{enumerate}
\end{definition}

In this work, our goal is to establish that a word or group is chiral if and only if it is $\gamma$-chiral for an anti-automorphism $\gamma$ if and only if it is $\gamma$-chiral for all anti-automorphism $\gamma$.

\section*{Main Result}
First we will prove a key lemma.

\begin{lemma}\label{lem}
    Let $F_d$ and $G$ be a free group with basis $x_1, x_2, \ldots, x_d$ and a group respectively. Let $\theta$ and $\zeta$ be automorphisms of $F_d$ and $G$ respectively. Then for a word $w \in F_d$, we have $G_w = G_{\theta(w)} = \zeta(G_w).$ 
\end{lemma}

\begin{proof}
    Let $w(g_1, g_2,\ldots, g_d) = g \in G_w$ for $g, g_1, g_2, \ldots, g_d \in G.$ It is easy to see that
    $$\theta(w(\theta^{-1}(x_1), \theta^{-1}(x_2), \ldots, \theta^{-1}(x_d)) = w.$$
    Hence $$\theta(w(\theta^{-1}(x_1)(g_1,\ldots, g_d), \theta^{-1}(x_2)(g_1, \ldots, g_d), \ldots, \theta^{-1}(x_d)(g_1, \ldots, g_d))) = w(g_1, \ldots, g_d).$$
    Therefore $G_w = G_{\theta(w)}$.
    
    We also have
    $$\zeta(w(\zeta^{-1}(g_1), \zeta^{-1}(g_2), \ldots, \zeta^{-1}(g_d)) = g.$$ Hence $\zeta(G_w) \subseteq G_w.$ Similarly  $\zeta^{-1}(G_w) \subseteq G_w.$ Applying $\zeta$ will give $\zeta(G_w) = G_w.$ Hence $G_w = G_{\theta(w)} = \zeta(G_w).$
\end{proof}
\begin{theorem}\label{thm: thm1}
 Let $w\in F_d,$ for some $d\in \N$ and $\gamma\in AA(F_d).$ Then $w$ is chiral if only if  $w$ is $\gamma$-chiral. 
\end{theorem}
\begin{proof}
First we start with a small observation. If  $\gamma\in AA(F_d),$
then there exists $\theta\in A(F_d)$ such that $\theta(w^{-1})=\gamma(w).$ But we have  $G_{\theta(w)}=G_w$ by lemma \ref{lem}. Hence, for every group $G,$ $G_{\gamma(w)}=G_{w^{-1}}.$

Suppose $w\in F_d$ is chiral. Then there is a group $G$ such that $G_w\ne G_{w^{-1}}.$ Consequently $G_w\ne G_{\gamma(w)}$ for every $\gamma\in AA(F_d).$ Hence $w$ is $\gamma$-chiral for every $\gamma\in AA(G).$

If the word is $\gamma$-chiral, then $G_w \neq G_{\gamma(w)}.$ But $G_{\gamma(w)} = G_{w^{-1}}$. This implies $G_w \neq G_{w^{-1}}$ proving that $w$ is chiral.
\end{proof}

Now we prove the equivalence for groups. We follow the same steps as in the above proof. Suppose $G$ is a group and $\gamma$ is an anti-automorphism of the group $G$. Then there exists an automorphism $\zeta$ of $G$ such that $\zeta(g^{-1}) = \gamma(g)$ for all  $g \in G$. Images of word maps are preserved by automorphism by lemma \ref{lem}, so we have $G_{w} = \zeta{(G_{w}})$.

Assume $G$ is $\gamma$-chiral, then there exists a word $w$ such that $G_w \ne \gamma(G_w).$ Thus $\gamma(G_w) = \zeta(G_w)^{-1} = \zeta(G_{w^{-1}}) = G_{w^{-1}}$.   This implies $G_w \ne G_{w^{-1}}$ proving $G$ is chiral. We can reverse the process to prove the converse. Hence we proved the following.
\begin{theorem}
 Let $G$ be a group and let $\ga\in AA(G).$ Then $G$ is chiral  if and only if $G$ is $\ga$-chiral. 
\end{theorem}

\begin{definition}
A finite group $G$ with an anti-automorphism $\gamma$ is said to be weakly $\gamma$-chiral if there exists $g \in G$ and $w \in F_d$ such that the cardinalities of fibres of $g$ with respect to word map $w$ and $w_\gamma$ are distinct. We say the pair is $(G,w)$ is weakly $\gamma$-chiral.   
\end{definition}

\begin{rem}
    It is worth noting that the property of weak $\gamma$-chirality is not dependent on the specific anti-automorphism $\gamma$ chosen.
\end{rem}

Gordeev et al. \cite{gordeev2018geometry} further posed an intriguing question whether there exists a group $G$ along with an anti-automorphism $\gamma$ of $G$ such that $G$ exhibits both $\gamma$-achirality and weakly $\gamma$-chirality. Notably, this question has been addressed by William Cocke \cite{William} for anti-automorphism defined by $\gamma(g) = g^{-1}$, but it is essential to emphasize that for all anti-automorphisms $\gamma$ considered, the property of weak $\gamma$-chirality remains invariant.

\bibliographystyle{plain}
\bibliography{SS_Zerowork}

\end{document}